\documentclass[english,12pt]{article}
\usepackage{amsmath,amsfonts,amssymb,babel,epic}    
\usepackage[active]{srcltx}
\newcounter{theoremcounter}
\newenvironment{theorem}{\refstepcounter{theoremcounter}{\vspace*{3mm}
   \par\noindent \bf Theorem \thetheoremcounter.}\it}{}

\newenvironment{lemma}{\refstepcounter{theoremcounter}{\vspace*{3mm}
   \par\noindent \bf Lemma \thetheoremcounter.}\it}{}

\renewcommand{\=}{\overset{ \text{def} }=}

\def\fsu#1#2#3{ {#1}_{#2}+\dots + {#1}_{#3}    }

\newcommand{\thmref}[1]{Theorem~\ref{#1}}

\newcommand{\lemref}[1]{Lemma~\ref{#1}}
\newcommand{\secref}[1]{Section~\ref{#1}}
\newenvironment{proof}[1][Proof]{\begin{trivlist}
\item[\hskip \labelsep {\bfseries #1}]}{\end{trivlist}}
\newcommand{\qed}{\nobreak \ifvmode \relax \else
      \ifdim\lastskip<1.5em \hskip-\lastskip
      \hskip1.5em plus0em minus0.5em \fi \nobreak
      \vrule height0.75em width0.5em depth0.25em\fi}

 \renewcommand{\qed}{\nobreak \ifvmode \relax \else
      \ifdim\lastskip<1.5em \hskip-\lastskip
      \hskip1.5em plus0em minus0.5em \fi \nobreak
      \vrule height0.75em width0.5em depth0.25em\fi}

\renewenvironment{proof}[1][Proof]{\begin{trivlist}
\item[\hskip \labelsep {\bfseries #1}]}{\end{trivlist}}

\renewcommand{\=}{\overset{ \text{def} }=}

\def\P{\mathbb{P}}
\def\Z{\mathbb{Z}}
\def\I{\mathbb{I}}

\def\R{\mathbb{R}}
\def\N{\mathbb{N}}
\def\M{\mathcal{M}}
\def\E{\mathbf{E}}
\def\F{\mathcal{F}}

\newcommand{\ve}{\varepsilon}

\begin{document}

\title{Random walks maximizing the probability to visit an interval}

\author{D. Dzindzalieta$^{1}$\\dainiusda[eta]gmail.com}

\footnotetext[1] {Vilnius University Institute of Mathematics and Informatics, 4 Akademijos, Vilnius, Lithuania. This research was funded by a grant (No. MIP-12090) from the Research Council of Lithuania}
\maketitle
\date{}
\begin{abstract}
We consider random walks, say $W_n=(M_0, M_1,\dots , M_n)$, of length $n$ starting at $0$
and based on the martingale sequence  $M_k$ with
differences $X_m=M_m-M_{m-1}$. Assuming that the differences are bounded,  $|X_m|\leq 1$, we solve the
problem
\begin{equation}
D_n(x)\=\sup\P\left\{ W_n \ \text{visits an interval }\ [x,\infty)\right\},\qquad x\in\R, \label{piirma}
\end{equation}
where $\sup$ is taken over all possible $W_n$. In particular, we describe random walks which maximize the probability
in $\eqref{piirma}$. We also extend the result to super-martingales.
\end{abstract}

\bigskip

\section{Introduction and results}

We consider random walks, say $W_n=(M_0, M_1,\ldots,M_n)$ of length $n$
starting at $0$
 and based on a martingale sequence  $M_k =\fsu X1k$ (assume $M_0=0$) with differences $X_m=M_m-M_{m-1}$.
Let $\mathcal M$ be the class of martingales with bounded
differences such that $|X_m|\leq 1$ and $\E(X_m|\F_{m-1})=0$ with respect to some increasing sequence of algebras $\emptyset\subset\F_0\subset\cdots\subset \F_n$. If a random walk $W_n$ is based on a
martingale sequence of the class $\mathcal M$
then we write symbolically
$W_n\in \mathcal M$. Extensions to super-martingales are provided at the end of the
Introduction.

In this paper we provide a solution of the problem
\begin{equation}
 D_n(x)\= \sup_{W_n\in \mathcal M}\, \P\left\{
W_n \ \text{visits an interval }\ [x,\infty)\right\}, \quad x\in \R
.\label{1.1}
\end{equation}

In particular, we describe random walks which maximize the
probability in $\eqref{1.1}$ and give an explicit expression
of the upper bound $D_n(x)$.
 It turns out that the random walk
maximizing the probability in $\eqref{1.1}$ is an inhomogeneous Markov chain, i.e., given $x$ and $n$, the distribution of $k$th step depends only on $M_{k-1}$ and $k$.
 For integer $x\in \Z $  the maximizing random walk is a simple symmetric random
walk (that is, a symmetric random walk with independent steps of
length $1$) stopped at $x$. For non-integer $x$, the maximizing random walk makes some steps of smaller sizes. Smaller steps are needed
to make a jump so that the remaining distance becomes integer and then continue as a simple random walk.
The average total number
 of the smaller steps is bounded by $2$. For martingales our result can be interpreted as
 a maximal  inequality
$$ \P \left\{\max_{1\leq k\leq n} M_k \geq x \right\}\leq D_n(x).$$

The maximal inequality is
  optimal  since the equality is achieved by martingales related to
the maximizing random walks, that is,
\begin{equation}
\sup_{W_1,\dots ,W_n \in \mathcal M}\,\P \left\{\max_{1\leq k\leq n} M_k \geq x \right\}=D_n(x),\label{1.a1}
\end{equation}
where we denote by $W_k$ a random walk $(M_0,M_1,\ldots,M_k,M_k,\ldots,M_k)\in \M$.

 To prove the result we introduce a general
principle for maximal inequalities for (natural  classes of)
martingales which reads as
\begin{equation}
 \sup_{W_1,\dots ,W_n \in \mathcal M}\,\P \left\{\max_{1\leq k\leq n} M_k \geq x \right\}
 =\sup_{M_n \in \mathcal M}\, \P \{ M_n \geq x \}\label{1.2}\end{equation}
 in our case. It means that for martingales, the solutions of problems of type \eqref{1.1} are \textit{inhomogeneous Markov chains}, i.e., the problem of type \eqref{1.1} can be always reduced to finding a solution of \eqref{1.1} in a class of inhomogeneous Markov chains.

Our methods are similar in
spirit to a method used in  \cite{bentkus2001inequality}, where a solution of a problem $\eqref{1.1}$
  was provided   for integer $x\in \Z$. Namely, he showed that if $R_n = \fsu \ve 1 n $ is a sum of Rademacher random variables such that $\P \{
\ve_i=-1\}=\P\{ \ve_i=-1\}=1/2$ and $B(n,k)$ is a normalized sum of $n - k + 1$ smallest
binomial coefficients, i.e.,
\begin{equation}\label{defbnk}
B(n,k)=2^{-n}\sum_{i=0}^{n-k}\binom{\left\lfloor\frac{i}{2}\right\rfloor}{n}
\end{equation}
where $\lfloor x \rfloor$ denotes an integer part of $x$, then for all $k\in\Z$
\begin{equation*}
D_n(k)=B(n,k)=\begin{cases}2\,\P\{R_n\geq k+1\}+\P\{R_n= k\}\quad &\text{if }n+k\in 2\Z,\\
2\,\P\{R_n\geq k+1\}\quad &\text{if }n+k\in 2\Z+1.
\end{cases}
\end{equation*}

Recently Dzindzalieta, Ju\v skevi\v cius and \v Sileikis $\cite{DJS2012}$ solved the problem $\eqref{1.1}$ in the case of sums of bounded independent symmetric random variables. They showed that if $~{S_n=X_1+\cdots+X_n}$ is a sum of independent symmetric random variables such that $|X_i|\leq1$  then
$$
\P\{S_n\geq x\}\leq\begin{cases}\P\{R_n\geq x\}\qquad &\text{if }n+\left\lceil x\right\rceil\in 2\Z,\\
\P\{R_{n-1}\geq  x\}\qquad &\text{if }n+\left\lceil x\right\rceil\in 2\Z+1,
\end{cases}
$$ where $\left\lceil x\right\rceil$ denotes the smallest integer number greater or equal to $x$. We note that for integer $x$  the random walk based on the sequence $R_k$ stopped at a level $x$ is a solution of $\eqref{1.1}$.

As  far as we are aware, the paper presents the first result
 where problems for martingales of type $\eqref{1.1}$ and $\eqref{1.a1}$ are solved for all $x\in \R$.

Let us turn to more detailed formulations of our results.
For a martingale $M_n\in\M$ and $x\in \R$, we introduce the
stopping time
\begin{equation}
 \tau_x = \min\{ k:\ M_k \geq x\}.\label{x.2222} \end{equation}
The stopping time $\tau_x$ is a non-negative integer valued random
variable possibly taking the value $+\infty$ in cases where
$M_k<x$ for all $k=0,1,\dots $. For a martingale $M _n\in \mathcal M
$, define its  version stopped at level $x$ as
\begin{equation}M_{n,x} = M_{\tau_x\wedge n}, \quad a\wedge b= \min \{a,b\}.\label{x.3}\end{equation}
Given a random walk $W_n=\{0,M_1,\ldots,M_n\}$ it's stopped version is denoted as
$W_{n,x}=\{0,M_{1,x},\dots ,
M_{ n,x}\}$.


Fix $n$ and $x>0$. The maximizing random walk $RW_n=\{0,M_1^x,\dots ,
M_n^x\}$ is defined as follows.  We start at $0$. Suppose that after $k$ steps the remaining distance to the target $[x,\infty)$ is $\rho_k$. The distribution of the next step is a Bernoulli random variable (which takes only two values), say $X^{\ast}=X^{\ast}(k,\rho_k,n)$, such that
\begin{equation}\label{eq:nextstep}
\sup\E D_{n-k}(\rho_k-X)=\E D_{n-k}(\rho_k-X^{\ast})
\end{equation} where $\sup$ is taken over all random variables $X$ such that $|X|\leq 1$ and $\E X=0$.

The distribution of the next step $X^{\ast}$ depends on four possible situations. \\
\\
\text{}\quad$i)$ $\rho_k$ is integer;\\
\text{}\quad$ii)$ $n-k$ is odd and $0<\rho_k<1$;\\
\text{}\quad$iii)$ the integer part of $\rho_k+n-k$ is even;\\
\text{}\quad$iv)$ the integer part of $\rho_k+n-k$ is odd and $\rho_k>1$.

\bigskip

After $k$ steps we make a step of length $s_l$ or $s_r$  to the left or right with probabilities $p_i=\frac{s_r}{s_r+s_l}$ and $q_i=\frac{s_l}{s_r+s_l}$ respectively. Let $\{x\}$ denotes the Depending on $(i)$-$(iv)$ we have.\\
\text{}\quad$i)$ $s_l=s_r=1$ with equal $~{\text{probabilities } p_1=q_1=\frac12}$, i.e., we continue as a simple random walk;\\
\text{}\quad$ii)$ $s_l=\rho_k$ and $s_r=1-\rho_k$  with $p_2=1-\{\rho_k\}$ and $q_2=\{\rho_k\}$, i.e., we make a step so that the remaining distance $\rho_{k+1}$ becomes equal either to $0$ or $1$;\\
\text{}\quad$(iii)$ $s_l=\{\rho_k\}$ and $s_r=1$ with $p_3=\frac{1}{1+\{\rho_k\}}$ and $q_3=\frac{\{\rho_k\}}{1+\{\rho_k\}}$, i.e., we make a step to the left so that $\rho_{k+1}$ is of the same parity as $n-k-1$ or to the right side as far as possible ;\\
\text{}\quad$(iv)$ $s_l=1$ and $s_r=1-\{\rho_k\}$ with $p_4=\frac{1-\{\rho_k\}}{2-\{\rho_k\}}$ and $q_4=\frac{1}{2-\{\rho_k\}}$, i.e., we make a step to the left so that $\rho_{k+1}$ is of the same parity as $n-k-1$ or to the right side as far as possible.

In other words if $\rho_k$ is non-integer then the maximizing random walk jumps so that $\rho_{k+1}$ becomes of the same parity as the remaining number of steps $n-k-1$ or the step of length $\min\{x,1\}$ to the other side. If the remaining distance $\rho_k$ is integer, then it continues as a simple random walk.

\bigskip

The main result of the paper is the following theorem.
\begin{theorem}\label{mainas}
The random walk $RW_{n}$ stopped at $x$ maximizes the probability to visit an interval $[x,\infty)$ in first $n$ steps, i.e.,
the following equalities hold
\begin{equation}
D_n(x)=\P\{RW_{n,x}\ \text{visits an interval }\ [x,\infty)\}=\P\{M_{n,x}^x\geq x\},\label{eq:mainas}
\end{equation} for all  $x\in\R$ and $n=0,1,2,\ldots$.
\end{theorem}

\bigskip

An explicit definition of $D_n(x)$ depends on the parity of $n$. Namely, let $x=m+\alpha$ with $m\in\Z$ and $0\leq\alpha<1$.

 If $m+n$ is odd then
\begin{equation}
D_n(x)=\sum_{i=0}^{h}a_i\,B(n-i-1,m+i), \quad a_i=\frac{\alpha^i}{(1+\alpha)^{i+1}},\label{eq:a3forma}
\end{equation} where $h=(n-m-1)/2$.

If $m+n$ is even then
\begin{equation}
D_n(x)=\sum_{i=0}^{m+1}b_i\, B(n-i-1,m-i+1),\label{eq:a4forma}
\end{equation}
where $b_i=\frac{(1-\alpha)^i}{(2-\alpha)^{i+1}}$, for $i<m$, $b_m=\alpha\left(\frac{1-\alpha}{2-\alpha}\right)^m$ and $b_{m+1}=(1-\alpha)\left(\frac{1-\alpha}{2-\alpha}\right)^m$.

\bigskip

It is easy to see from $\eqref{eq:a3forma}$ and $\eqref{eq:a4forma}$ that $D_n$ is decreasing and continuous for all $x\in\R$ except at $x=n$ it has a jump. In particular we have that $D_n(x)=1$ for $x\leq 0$ and $D_n(x)=0$ for $x>n$.
In $\secref{sec:proofs}$ we prove that the function $D_n$ is piecewise convex and piecewise continuously differentiable. We also give the recursive definition of the function $D_n$.

\bigskip

A great number of papers is devoted to construction of {\it upper}
bounds for tail probabilities of sums of random variables. The reader can find classical results in books \cite{petrov1975sums, shorack2009empirical}. One of the first and probably the most known non-asymptotic bound for $D_n(x)$ was given by Hoeffding in 1963 \cite{hoeffding1963probability}.
He proved that for all $x$ the function $D_n(x)$ is bounded by $\exp\{-x^2/2n\}$. Hoeffding's inequalities remained unimproved until 1995 when Talagrand \cite{talagrand1995missing} inserted certain missing factors. Bentkus 1986--2007 \cite{bentkus87banach, bentkus2001inequality,bentkus2004hoeffding, bentkus2006domination} developed
induction based methods. If it is possible to overcome related technical difficulties, these methods
lead to the best known upper bounds for the tail probabilities (see \cite{BD2010, DJS2012} for examples of tight bounds received using these methods). In \cite{bentkus2001inequality}
first tight bounds for $D_n(x)$ for integer $x$ was received. To overcome technical difficulties for non-integer $x$ in \cite{bentkus2001inequality} the linear interpolation between integer points was used, thus losing precision for non-integer $x$. Our method is similar in spirit to \cite{bentkus2001inequality}.

\subsection{An extension to super-martingales}
Let $\mathcal{SM}$ be the class of super-martingales with bounded differences such
that $|X_m|\leq 1$ and $~{\E (X_k|\F_{k-1})\leq 0}$ with respect to some increasing sequence of algebras $\emptyset\subset\F_0\subset\cdots\subset \F_n$. We show that
\begin{theorem}\label{th:supermart} For all $x\in R$ we have
\begin{equation}
\sup_{SW_n \in\, \mathcal{SM}}\, \P\left\{ SW_n \
\text{visits an interval }\ [x,\infty)\right\}=D_n(x).\label{ineq:superm}
\end{equation}
\end{theorem}

For super-martingales $\thmref{th:supermart}$ can also be
interpreted as the maximal inequality

\begin{equation*} \P \left\{\max_{1\leq k\leq n} M_k \geq x \right\}\leq D_n(x), \end{equation*}
where $M_k \in \mathcal {SM}$, and furthermore, the $\sup$
over the class of super-martingales is achieved on a martingale
class.

\textbf{Proof of $\thmref{th:supermart}$.} Suppose that $\sup$ in $\eqref{ineq:superm}$ is achieved with some super-martingale $SM_n=X_1+\cdots+X_n$. Let $M_n=Y_1+\cdots+Y_n$ be a sum of random variables, such that $$(Y_k|\F_{k-1})=((X_k|\F_{k-1})-1)\frac{\E(X_k|\F_{k-1})}{1-\E(X_k|\F_{k-1})}.$$ It is easy to see that $Y_k\geq 0$, $|X_k+Y_k|\leq 1$ and  $\E(X_k+Y_k|\F_{k-1})=0$, so $SM_n+M_n\in\mathcal M$.
Since $Y_k\geq 0$ we have that $M_n\geq 0$, so $\P\left\{ SM_n+M_n\geq x\right\}$ is greater or equal to $\P\left\{ SM_n\geq x\right\}$. This proves the theorem.~$\square$
\section{Maximal inequalities for martingales are equivalent to
inequalities for tail probabilities}

Let $\mathcal M $ be a class of martingales. Introduce the upper
bounds for tail probabilities and in the maximal inequalities as
\begin{equation*}
 B_n(x)\=\sup_{M_n \in \mathcal M} \P  \{ M_n \geq x \},
\quad B_n^\ast (x)\=\sup_{M_n \in \mathcal M} \P \left\{ \max_{0\leq
k\leq n }M_k \geq x \right\}\end{equation*}
for $x\in \R$ (we define $M_0=0$).

Let  as before $\tau_x$ be a stopping time defined by
\begin{equation}
 \tau_x = \min\{ k:\ M_k \geq x\}.\end{equation}

\begin{theorem}\label{x.1}
If a class $\mathcal M$ of martingales is
closed under stopping at level $x$, then
$$B_n(x)\equiv B_n^\ast (x).$$
\end{theorem}

We can interpret $\thmref{x.1}$ by saying that inequalities for tail
probabilities for natural classes of martingales imply (seemingly
stronger) maximal inequalities. This means that maximizing martingales are \textit{inhomogeneous Markov chains}. Assume that for all $M_n\in \mathcal
M$ we have
$$\P \{ M_n\geq x\} \leq g_n(x)$$
 with some function $g$ which depends only on $n$ and the class $\mathcal M$. Then it follows that
$$\P \left\{ \max_{0\leq k\leq n} M_k\geq x\right\} \leq g_n(x).$$
In particular, equalities $\eqref{1.1}$--$\eqref{1.2}$ are equivalent.

\textbf{Proof of $\thmref{x.1}$.} It is clear that $B_n\leq B_n^\ast$
since $M_n\leq \max\limits_{0\leq k\leq n} M_k$. Therefore it
suffices to check the opposite inequality $B_n\geq B_n^\ast$. Let
$M_n\in \mathcal M$. Using the fact that $M_{\tau_x\wedge n}\in \mathcal M$, we have
\begin{equation}
 \P \left\{ \max_{0\leq k\leq n} M_k\geq x\right\} = \P \left\{ M_{\tau_x\wedge n}\geq x\right\}\leq B_n(x). \label{x.34}
 \end{equation}
Taking in $\eqref{x.34}$ $\sup$ over $M_n\in \mathcal M$, we derive $B_n^\ast
\geq B_n$. ~$\square$

\bigskip

In general conditions of $\thmref{x.1}$ are fulfilled under usual moment and range conditions.
That is, conditions of type
$$\E \left( |X_k|^{\alpha_k}\, \left|\, \mathcal F_{k-1}\right. \right)\leq g_k,\quad \left(X_k\, \left|\, \mathcal F_{k-1}\right. \right)\in I_k,$$
with some $\mathcal F_{k-1}$-measurable $\alpha_k\geq 0$, $g_k\geq 0$,
and intervals $I_k$ with $\mathcal F_{k-1}$-measurable endpoints.
One can use as well assumptions like symmetry, unimodality, etc.

\bigskip

\section{Proofs}\label{sec:proofs}
In order to prove $\thmref{mainas}$ we need some additional lemmas.
\begin{lemma}\label{lem:convx}
Suppose $f\in C^1(0,2)$ is a continuously differentiable, non-increasing, convex function on $(0,2)$. Suppose that $f$ is also two times differentiable on intervals $(0,1)$ and $(1,2)$. The function $F:(0,2)\rightarrow R$ defined as
\begin{eqnarray*}
F(x)&=&\frac{1}{x+1}f(0)+\frac{x}{x+1}f(x+1)\qquad\text{for }\,x\in(0,1];\\
F(x)&=&\frac{2-x}{3-x}f(x-1)+\frac{1}{3-x}f(2)\qquad\text{for }\,x\in(1,2)
\end{eqnarray*}
is convex on intervals $(0,1)$ and $(1,2)$.
\end{lemma}
\begin{proof}
Since the function $f$ is decreasing and convex, we have that
\begin{eqnarray}
f^{\prime}(x+1)&\geq&\frac{f(x+1)-f(0)}{x+1}\qquad\text{for }\,x\in(0,1);\label{pirma}\\
f^{\prime}(x-1)&\leq&\frac{f(2)-f(x-1)}{3-x}\qquad\text{for }\,x\in(1,2).\label{antr}
\end{eqnarray}

For $x\in(0,1)$ simple algebraic manipulations gives \begin{equation}
F^{\prime\prime}(x)= \frac{x}{x+1}f^{\prime\prime}(x+1)+\frac{2}{(x+1)^2}\left(f^{\prime}(x+1)-\frac{f(x+1)-f(0)}{x+1}\right).\label{pirma2}
\end{equation}
By $\eqref{pirma}$ the second term in right hand side of $\eqref{pirma2}$ is non-negative.
Thus $F^{\prime\prime}(x)\geq0$ for all $x\in (0,1)$.

For $x\in(1,2)$ similar algebraic manipulation gives
\begin{equation}F^{\prime\prime}(x)= \frac{2-x}{3-x}f^{\prime\prime}(x-1)-\frac{2}{(3-x)^2}\left(f^{\prime}(x-1)-\frac{f(2)-f(x-1)}{3-x}\right).\label{antr2}
\end{equation}
By $\eqref{antr}$ the second term in right hand side of $\eqref{antr2}$ is non-negative.
Thus $F^{\prime\prime}(x)\geq0$ for all $x\in (1,2)$.\qed
\end{proof}

We use $\lemref{lem:convx}$ to prove that the function $x\rightarrow D_n(x)$ satisfies the following analytic properties.
\begin{lemma}\label{lem:dprop}
 The function $D_n$ is convex and continuously differentiable on intervals $~{(n-2,n), (n-4,n-2),\dots, (0,2\{n/2\})}$.
\end{lemma}

\begin{proof}
In order to prove this lemma it is very convenient to use a recursive definition of the function $D_n(x)$ which easily follows from the the description of the maximizing random walk $RW_{n,x}$. We have $D_0(x)=\I\{x\leq0\}$ and
\begin{equation}\label{eq:dnxrec}
D_{n+1}(x)=\begin{cases}
1 & \text{if }\,x\leq0,
\\
p_1D_{n}(x-1)+q_1D_{n}(x+1) & \text{if } x\in\Z \text{ and } x>0,
\\
p_2D_{n}(0)+q_2D_{n}(1) &\text{if } n\in 2\Z+1\text{ and } x<1,
\\
p_3D_{n}(\left\lfloor x\right\rfloor)+q_3D_{n}(x+1) &\text{if } \lfloor x\rfloor+n\in 2\Z \text{ and } x>0,
\\
p_4D_{n}(x-1)+q_4D_{n}(\left\lceil x\right\rceil) &\text{if }\lceil x\rceil+n\in 2\Z\text{ and } x>1,
\\
0 & \text{if }x>n.
\end{cases}
\end{equation}
where $p_i+q_1=1$ with $p_1=1/2$, $p_2=1-\{x\}$, $p_3=\frac1{1+\{x\}}$ and $p_4=\frac{1-\{x\}}{2-\{x\}}$.

\bigskip
To prove $\lemref{lem:convx}$ we use induction on $n$. If $n=0$ then $~{D_n(x)=\I\{x\leq 0\}}$ clearly satisfies \lemref{lem:dprop}.
Suppose that $\lemref{lem:convx}$ holds for $n=k-1\geq0$. Assume $n=k$.

First we prove that $D_k$ is convex and continuously differentiable on intervals $(0,1),(1,2),\ldots,(k-1,k)$. Since $D_k$ is rational and do not have discontinuities between integer points, it is clearly continuously differentiable on intervals $(0,1),(1,2),\ldots,(k-1,k)$. If $x\in(k-1,k]$ then by \eqref{eq:a3forma} we have that
$~{D_k(x)=2^{-k+1}/(x-k+1)}$. Thus the function $D_k$ is clearly convex on interval $(k-1,k)$. The convexity of $D_k$ on intervals $(0,1),(1,2),\ldots,(k-2,k-1)$ follows directly from $\lemref{lem:convx}$ and recursive definition $\eqref{eq:dnxrec}$.
To prove that the function $D_k$ is also continuously differentiable on intervals $~{(k-2,k), (k-4,k-2),\dots, (0,2\{k/2\})}$ it is enough to show that $D^{\prime}_k(m-0)=D^{\prime}_k(m+0)$ for all $m\in \N$ such that $k+m\in 2\Z+1$. If $x=m-0$ (we consider only the case $m>0$, since for $m=0$ the function $D_k(x)$ is linear), then by $\eqref{eq:dnxrec}$ we have
\begin{equation}D_k(x)=p_4D_{k-1}(x-1)+q_4D_{k-1}(m)\end{equation}
and since $D_{k-1}$ is continuously differentiable at $x-1$ we have $$D_k^{\prime}(x)=q_4^2D_{k-1}(x-1)+p_4D^{\prime}_{k-1}(x-1)
-q_4^2D_{k-1}(m).$$ Since $x=m-0$ we get that $D_k^{\prime}(x)=D_{k-1}(x-1)-D_{k-1}(m)$.
Similarly we have that if $x=m+0$ then $D_k^{\prime}(x)=D_{k-1}(m)-D_{k-1}(x-1)$. Since $D_{k-1}(m-1)-D_{k-1}(m)=D_{k-1}(m)-D_{k-1}(m-1)$ we get that  $D^{\prime}_k(m-0)=D^{\prime}_k(m+0)$.
Since $D_k$ is continuously differentiable on intervals $~{(k-2,k), (k-4,k-2),\dots, (0,2\{k/2\})}$ and $D_k$ is convex on intervals $(0,1),(1,2),\ldots,(k-1,k)$ we have that $D_k(x)\geq0$ is convex on $x=m$ for all $m\in\N$ such that $~{k+m\in 2\Z+1}$. This ends the proof of $\lemref{lem:convx}$.
\qed
\end{proof}

We also need the following lemma, which is used to find the minimal dominating linear function in a proof of \thmref{mainas}.
\begin{lemma}\label{lem:dnelyg} The function $D_n$ satisfies the following inequalities.\\
a) If $n\in2\Z+1$ and $0<x<1$ then
\begin{equation}p_2D_{n}(0)+q_2D_{n}(1)- p_3D_n(0)-q_3D_n(x+1)\geq 0.\label{eq:dnely.1stine}\end{equation}
b) If $\lfloor n+x\rfloor\in2\Z$ then
\begin{equation}p_3D_{n}(\left\lfloor x\right\rfloor)+q_3D_{n}(x+1)- p_1D_n(x-1)-q_1D_n(x+1)\geq0.\label{eq:dnely.2ndine}\end{equation}
c) If $\lfloor n+x\rfloor\in2\Z+1$ and $x>1$
\begin{equation}p_4D_{n}(x-1)+q_4D_{n}(\left\lfloor x\right\rfloor+1)- p_1D_n(x-1)-q_1D_n(x+1)\geq0.\label{eq:dnely.3rdine}\end{equation}
\end{lemma}
Here $p_i$ and $q_i$ are the same as in $\lemref{lem:convx}$.
\begin{proof}
We prove this lemma by induction on $n$. If $n=0$ then $\lemref{lem:dnelyg}$ is equivalent to the trivial inequality $1-1\geq0$.
Suppose that the properties $(a)$--$(c)$ holds for $n=k-1\geq0$. Assume $n=k$.

\textbf{Proof of (a).} We use the following equalities directly following from the definition of the function $D_k$.
If $k\in 2\Z+1$ and $x\in (0,1)$ then
\begin{eqnarray*}
D_k(0)&=&D_{k-1}(0);\\
D_k(1)&=&p_1D_{k-1}(0)+q_1D_{k-1}(2);\\
D_k(x+1)&=&p_4D_{k-1}(2)+q_4D_{k-1}(x);\\
D_{k-1}(x)&=&D_{k-1}(0)+x\left(D_{k-1}(1)-D_{k-1}(0)\right);
\end{eqnarray*}
We substitute all these equalities to $\eqref{eq:dnely.1stine}$ we get that the left hand side of $\eqref{eq:dnely.1stine}$ is equal to
\begin{equation*}
q_2p_3p_4x\left(p_1D_{k-1}(0)+q_1D_{k-1}(2)-D_{k-1}(1)\right)
\end{equation*}
The inequality $\eqref{eq:dnely.1stine}$ follows from the inequality $D_{k-1}(1)\leq D_{k}(1)=p_1D_{k-1}(0)+q_1D_{k-1}(2)$.

\textbf{Proof of (b).}
We rewrite every term in the inequality $\eqref{eq:dnely.2ndine}$ using the definition of the function $D_k$ to get
\begin{eqnarray*}
&&p_1 p_3(D_{k-1}(\left\lfloor x \right\rfloor-1 )+D_{k-1}(\left\lfloor x \right\rfloor+1 ))+q_3(p_3D_{k-1}(\left\lfloor x \right\rfloor+1 )+q_3D_{k-1}(x+2 ))-\\
&&p_1p_3(D_{k-1}(\left\lfloor x \right\rfloor-1 )+D_{k-1}(\left\lfloor x \right\rfloor+1 ))-p_1q_3(D_{k-1}(x)+D_{k-1}(x+2 ))\geq0.
\end{eqnarray*}
The inequality
$$
p_3D_{k-1}(\left\lfloor x \right\rfloor+1 )+q_3D_{k-1}(x+2 )\geq p_1D_{k-1}(x)+q_1D_{k-1}(x+2 )
$$
follows from the inductive assumption \eqref{eq:dnely.2ndine} for $n=k-1$.

\textbf{Proof of (c).} In this case we have to consider two separate cases.

\textbf{Case $x>2$}. We again rewrite every term in the inequality $\eqref{eq:dnely.3rdine}$ using the definition of the function $D_k$ to get
\begin{eqnarray*}
&&p_4(p_4D_{k-1}(x-2 )+q_4D_{k-1}(\left\lfloor x \right\rfloor ))+q_4p_1(D_{k-1}(\left\lfloor x \right\rfloor)+D_{k-1}(\left\lfloor x \right\rfloor+2 ))-\\
&&q_4p_1(D_{k-1}(\left\lfloor x \right\rfloor )+D_{k-1}(\left\lfloor x \right\rfloor+2 ))-p_4p_1(D_{k-1}(x-2)+D_{k-1}(x))\geq0.
\end{eqnarray*}
The inequality
$$
p_4D_{k-1}(x-2 )+q_4D_{k-1}(\left\lfloor x \right\rfloor )\geq p_1D_{k-1}(x-2)+q_1D_{k-1}(x)
$$
follows from the inductive assumption \eqref{eq:dnely.3rdine} for $n=k-1$.

\textbf{Case $1<x<2$}. Firstly let us again rewrite the inequality $\eqref{eq:dnely.3rdine}$ using the recursive definition of $D_k$. After combining the terms we get that $\eqref{eq:dnely.3rdine}$ is equivalent to
\begin{equation}
 x\,D_{{k-1}}
 \left( 1 \right) +(1-x)D_{{k-1}} \left( 0 \right) -D_{{k-1}} \left(
x \right) \geq0.
\end{equation}
Now we use the inequality
$$
D_{k-1}(x)\leq (2-x)D_{k-1}(1)+(x-1)D_{k-1}(2).
$$
to get that
\begin{eqnarray*}
&&xD_{k-1}(1)+(1-x)D_{k-1}(0)-D_{k-1}(x)\geq\\
&&2(x-1)D_{k-1}(1)+(1-x)D_{k-1}(0)-(x-1)D_{k-1}(2)=\\
&&(x-1)\left(2D_{k-1}(1)-D_{k-1}(0)-D_{k-1}(2)\right)=0.
\end{eqnarray*} which proves the inequality $\eqref{eq:dnely.3rdine}$.\qed
\end{proof}

Now we are ready to prove $\thmref{mainas}$.
\begin{proof}
For $x\leq 0$ to achieve $\sup$ in $\eqref{eq:mainas}$ take $M_n\equiv 0$. For $x>n$ the $\sup$ in $\eqref{eq:mainas}$ is equal to zero since $M_n \leq n$ for all $n=0,1,\ldots$.
To prove $\thmref{mainas}$ for $x\in(0,n]$ we use induction on $n$.

For $n=0$ the statement is obvious since $\P\{M_0\geq x\}=\I\{x\leq 0\}=D_0(x)$.
Suppose that $\thmref{mainas}$ holds for $n=k>0$.
Assume $n=k+1$.
In order to prove $\thmref{mainas}$ it is enough to prove that $D_{k+1}$ satisfies the recursive relations $\eqref{eq:dnxrec}$.
We have
\begin{eqnarray*}
\P\{M_{k+1}\geq x\}&=&\P\{X_2+\cdots+X_{k+1}\geq x-X_{1}\}\\
&=&\E\P\{X_2+\cdots+X_{k}\geq x-t|X_1=t\}\\&\leq& \E D_{k}(x-X_1).
\end{eqnarray*}

Now for every $x$ we find a linear function $t\mapsto f(t)$ dominating the function $t\mapsto D_k(x-t)$ on interval $[-1,1]$ and touching it at two points, say $x_1$ and $x_2$, on different sides of zero. After this we consider a random variable, say $X\in\{x_1,x_2\}$ with mean zero. It is clear that $\E D_{k}(x-X_1)\leq \E D_{k}(x-X)$. We show that the numbers $x_1$ and $x_2$ are so that $\eqref{eq:dnxrec}$ holds.

Since $D_k$ is piecewise convex between integer points, the points where $f(t)$ touches $D_k(x-t)$ can be only the endpoints of an interval $[-1,1]$ or the points where $D_k(x-t)$ is not convex.

We consider four separate cases.\begin{eqnarray*}
&i)&\, x\in\Z;\\
&ii)&  k\in2\Z+1\quad \text{and}\quad x<1;\\
&iii)& \lfloor x\rfloor+k\in2\Z;\\
&iv)& \lfloor x\rfloor+k\in2\Z\quad \text{and}\quad x>1.
\end{eqnarray*}
\textbf{Case $(i)$.} Since $x\in\Z$ the dominating linear function touches $D_k(x-t)$ at integer points. So maximizing $X_1\in\{-1,0,1\}$.\\
If $x+k\in 2\Z+1$ then the function $D_k(x-t)$ is convex on $(-1,1)$ so maximizing $X$ takes values $1$ or $-1$ with equal probabilities~$1/2$.\\
If $x+k\in 2\Z$, then $$D_k(x)=\frac12\left(D_{k-1}(x-1)+D_{k-1}\left(x+1\right)\right)=\frac12(D_k(x-1)+D_k(x+1)),$$ so the dominating function touches $D_k(x-t)$ at all three points $-1,0,1$. Taking $X\in\{-1,1\}$ we end the proof of the case $(i)$.

The case $(i)$ was firstly considered in $\cite{bentkus2001inequality}$.

\textbf{Case $(ii)$.} Since $D_k$ is convex on intervals $(0,1)$ and $(1,3)$ the dominating minimal function can touch $D_k(x-t)$ only at $x,x-1,-1$. But due to an inequality $\eqref{eq:dnely.1stine}$ the linear function $f(t)$ going through $(x,D_k(0))$ and $(x-1,D_k(1))$ is above the point $(-1,D_k(x+1))$.

\textbf{Case $(iii)$.} Since the function $D_k$ is convex on intervals $(\left\lfloor x\right\rfloor-1,\left\lfloor x\right\rfloor)$ and $(\left\lfloor x\right\rfloor,\left\lfloor x\right\rfloor+2)$ the dominating minimal function can touch $D_k(x-t)$ only at $-1,\left\{ x\right\},1$. But due to an inequality $\eqref{eq:dnely.2ndine}$ the linear function $f(t)$ going through $(\left\{ x\right\},D_k(\left\lfloor x\right\rfloor))$ and $(-1,D_k(x+1))$ is above the point $(1,D_k(x-1))$.

\textbf{Case $(iv)$.} Since the function $D_k$ is convex on intervals $(\left\lfloor x\right\rfloor-1,\left\lfloor x\right\rfloor+1)$ and $(\left\lfloor x\right\rfloor+1,\left\lfloor x\right\rfloor+3)$ the dominating minimal function can touch $D_k(x+t)$ only at $-1,\left\{x\right\}-1,1$. But due to an inequality $\eqref{eq:dnely.3rdine}$ the linear function $f(t)$ going through $(1,D_k(x-1))$ and $(\left\{x\right\}-1,D_k(\left\lfloor x\right\rfloor+1))$ is above the point $(-1,D_k(x+1))$.\qed
\end{proof}
\bibliographystyle{alpha}

\begin{thebibliography}{30}

\bibitem[Ben87]{bentkus87banach}
V.~Bentkus.
\newblock Large deviations in Banach spaces.
\newblock {\em Theory of Probability \& Its Applications}, 31(4):627--632, 1987.

\bibitem[Ben01]{bentkus2001inequality}
V.~Bentkus.
\newblock An inequality for large deviation probabilities of sums of bounded
  iid random variables.
\newblock {\em Lithuanian Mathematical Journal}, 41(2):112--119, 2001.

\bibitem[Ben04]{bentkus2004hoeffding}
V.~Bentkus.
\newblock On hoeffding's inequalities.
\newblock {\em Annals of probability}, pages 1650--1673, 2004.

\bibitem[BKZ06]{bentkus2006domination}
V.~Bentkus, N.~Kalosha, and M.~Van~Zuijlen.
\newblock On domination of tail probabilities of (super) martingales: explicit
  bounds.
\newblock {\em Lithuanian Mathematical Journal}, 46(1):1--43, 2006.

\bibitem[BD10]{BD2010}
D.~Dzindzalieta and V.~Bentkus.
\newblock A tight Gaussian bound for weighted sums of Rademacher random variables.
\newblock {\em preprint}, 2010.



\bibitem[DJS12]{DJS2012}
D.~Dzindzalieta, T.~Ju\v skevi\v cius, and M.\v Sileikis.
\newblock Optimal probability inequalities for random walks related to problems
  in extremal combinatorics.
\newblock {\em SIAM Journal on Discrete Mathematics}, 26(2):828--837, 2012.

\bibitem[Hoe63]{hoeffding1963probability}
W.~Hoeffding.
\newblock Probability inequalities for sums of bounded random variables.
\newblock {\em Journal of the American Statistical Association},
  58(301):13--30, 1963.

\bibitem[KS66]{karlin1966tchebycheff}
S.~Karlin and W.J. Studden.
\newblock {\em Tchebycheff systems: With applications in analysis and
  statistics}, volume 376.
\newblock Interscience Publishers New York, 1966.

\bibitem[PB75]{petrov1975sums}
V.V. Petrov and A.A. Brown.
\newblock {\em Sums of independent random variables}, volume 197-5.
\newblock Springer-Verlag Berlin, 1975.

\bibitem[SW09]{shorack2009empirical}
G.R. Shorack and J.A. Wellner.
\newblock {\em Empirical processes with applications to statistics}, volume~59.
\newblock Society for Industrial Mathematics, 2009.

\bibitem[Tal95]{talagrand1995missing}
M.~Talagrand.
\newblock The missing factor in hoeffding's inequalities.
\newblock In {\em Annales de l'IHP Probabilit{\'e}s et statistiques}, volume
  31-4, pages 689--702. Elsevier, 1995.

\end{thebibliography}

\end{document}